\documentclass{amsart}
\usepackage{amssymb,amsmath,amsfonts,latexsym,amsthm,geometry}
\usepackage{amsmath}
\usepackage{amsfonts}
\usepackage{amssymb}
\usepackage{graphicx}

\providecommand{\U}[1]{\protect\rule{.1in}{.1in}}
\providecommand{\U}[1]{\protect\rule{.1in}{.1in}}
\providecommand{\U}[1]{\protect\rule{.1in}{.1in}}
\newtheorem{theorem}{Theorem}[section]
\newtheorem{corollary}[theorem]{Corollary}

\newtheorem{lemma}[theorem]{Lemma}

\theoremstyle{definition}

\newtheorem{example}[theorem]{Example}

\newtheorem{remark}[theorem]{Remark}
\newtheorem{definition}[theorem]{Definition}

\begin{document}
\title[On the size of certain subsets of invariant Banach sequence spaces]{%
On the size of certain subsets of invariant Banach sequence spaces}
\author[T.K. Nogueira]{Tony Nogueira}
\address{Departamento de Matem\'{a}tica, Universidade Federal da Para\'{\i}%
ba, 58.051-900 - Jo\~{a}o Pessoa, Brazil.}
\email{tonykleverson@gmail.com}
\author[D. Pellegrino]{Daniel Pellegrino}
\address{Departamento de Matem\'{a}tica, Universidade Federal da Para\'{\i}%
ba, 58.051-900 - Jo\~{a}o Pessoa, Brazil.}
\email{pellegrino@pq.cnpq.br and dmpellegrino@gmail.com}
\thanks{The authors are supported by CNPq and Capes}

\begin{abstract}
The essence of the notion of lineability and spaceability is to find linear
structures in somewhat chaotic environments. The existing methods, in
general, use \textit{ad hoc} arguments and few general techniques are known.
Motivated by the search of general methods, in this paper we formally extend
recent results of G.\ Botelho and V.V. F\'{a}varo on invariant sequence
spaces to a more general setting. Our main results show that some subsets of
invariant sequence spaces contain, up to the null vector, a closed
infinite-dimensional subspace.
\end{abstract}

\maketitle

\section{Introduction and background}

The notion of invariant sequence spaces, as we investigate in this note, was
introduced in \cite{bdf} although it seems to have its roots in \cite{barr,
bc}. Our main results are formal extensions of recent results of G. Botelho
and V.V. F\'{a}varo \cite{favv}. We show, among other results, that some
special invariant sequence spaces used in \cite{favv} can be replaced by
more general invariant sequence spaces. Let us first recall the notion of
invariant sequence space.

\begin{definition}
\label{seqspa} (\cite{bdf}) Let $X\neq \{0\}$ be a Banach space.\newline
(a) Given $x\in X^{\mathbb{N}}$, $x^{0}$ is defined as: if $x$ has only
finitely many non-zero coordinates, then $x^{0}=0$; otherwise, $%
x^{0}=(x_{j})_{j=1}^{\infty }$ where $x_{j}$ is the $j$-th non-zero
coordinate of $x$.\newline
(b) An invariant sequence space over $X$ is an infinite-dimensional Banach
or quasi-Banach space $E$ of $X$-valued sequences enjoying the following
conditions:\newline
(b1) For $x\in X^{\mathbb{N}}$ such that $x^{0}\neq 0$, $x\in E$ if and only
if $x^{0}\in E$, and $\Vert x\Vert \leq K\Vert x^{0}\Vert $ for some
constant $K$ depending only on $E$.\newline
(b2) $\Vert x_{j}\Vert \leq \Vert x\Vert $ for every $x=(x_{j})_{j=1}^{%
\infty }\in E$ and every $j\in \mathbb{N}$. \newline
\end{definition}

\begin{example}
As mentioned in \cite{bdf}, usual sequence spaces are invariant sequence
spaces. For instance

(a) For every $0<p\leq \infty ,$ the spaces 
\begin{equation*}
\begin{aligned} \ell _{p}(X) &= \left\{ (x_{j})_{j=1}^{\infty }\in
X^{\mathbb{N}}:\Vert (x_{j})_{j=1}^{\infty }\Vert:=\left(
\displaystyle\sum_{j=1}^{\infty }\Vert x_{j}\Vert ^{p}\right) ^{\frac{1}{p}
}<\infty \right\} , \\ \ell _{p}^{w}(X)& =\left\{ (x_{j})_{j=1}^{\infty }\in
X^{\mathbb{N}}:\Vert (x_{j})_{j=1}^{\infty
}\Vert_{w,p}:=\displaystyle\sup_{\|\varphi\|\leq
1}\left(\displaystyle\sum_{j=1}^{\infty } \left\vert \varphi
(x_{j})\right\vert^{p}\right) ^{\frac{1}{p}}<\infty ,\varphi \in X^{\prime
}\right\} , \\ \ell _{p}^{u}(X) & =\left\{ (x_{j})_{j=1}^{\infty }\in \ell
_{p}^{w}(X):\displaystyle\lim_{n\rightarrow \infty }\Vert
(x_{j})_{j=n}^{\infty }\Vert_{w,p}=0\right\} , \\ c_{0}(X)& =\left\{
(x_{j})_{j=1}^{\infty }\in X^{\mathbb{N}}:\lim_{j\rightarrow \infty
}x_{j}=0\right\}, \\ c(X) & =\left\{ (x_{j})_{j=1}^{\infty }\in
X^{\mathbb{N}}:\lim_{j\rightarrow \infty }x_{j}\mbox{ exists}\right\}
\end{aligned}
\end{equation*}%
are invariant sequence spaces over $X$. Above and henceforth, $X^{\prime }$
denotes the topological dual of $X,$ and in $c_{0}(X)$ and $c(X)$ we
consider the $\sup $ norm. When $p=\infty ,$ the sums are replaced by a
supremum and $\ell _{\infty }^{w}(X):=\ell _{\infty }(X).$\newline
(b) For $0<p,q<\infty $, the Lorentz space $\ell _{p,q}$ is an invariant
sequence space (over $\mathbb{K}$). For more details and examples we refer
to \cite{bdf}.
\end{example}


The spirit of the concept of lineability and spaceability is to look for
linear structures in nonlinear settings. There are few general methods (see,
for instance, \cite{b2, bbb}) to prove lineability and spaceability and, in
general, particular problems need \textit{ad hoc} arguments. For more
details on the subject we refer to \cite{aron, vv, bernal} and the
references therein.

The following definition is a natural extension of \cite[Definition 2.2]%
{favv}:

\begin{definition}
Let $X$ and $Y$ be Banach spaces, $\Gamma $ be an arbitrary set and $E$ be
an invariant sequence space over $X.$ If $E_{l},$ for all $l\in \Gamma ,$
are invariant sequence spaces over $Y$ and $f:X\longrightarrow Y$ is any
map, we define the set%
\begin{equation*}
G(E,f,(E_{l})_{l\in \Gamma })=\left\{ (x_{j})_{j=1}^{\infty }\in
E:(f(x_{j}))_{j=1}^{\infty }\notin \bigcup_{l\in \Gamma }E_{l}\right\} .
\end{equation*}
\end{definition}

According to \cite[Definition 2.3]{favv} a map $f\colon X\longrightarrow Y$
between normed spaces is said to be: \newline
(a) \textit{Non-contractive} if $f(0)=0$ and for every scalar $\alpha \neq 0$
there is a constant $K(\alpha )>0$ such that 
\begin{equation}
\Vert f(\alpha x)\Vert _{Y}\geq K(\alpha )\cdot \Vert f(x)\Vert _{Y}
\label{yjj}
\end{equation}%
for every $x\in X$. \newline
(b) \textit{Strongly non-contractive} if $f(0)=0$ and for every scalar $%
\alpha \neq 0$ there is a constant $K(\alpha )>0$ such that 
\begin{equation}
|\varphi (f(\alpha x))|\geq K(\alpha )\cdot |\varphi (f(x))|  \label{dd44}
\end{equation}%
for all $x\in X$ and $\varphi \in Y^{\prime }$. \newline
The following result was recently proved by Botelho and F\'{a}varo (see \cite%
[Theorem 2.5]{favv}):

\begin{theorem}
\label{u8} \textrm{\textrm{(\cite{favv})}} Let $X$ and $Y$ be Banach spaces, 
$E$ be an invariant sequence space over $X$, $f\colon X\longrightarrow Y$ be
a function and $\Gamma \subseteq (0,\infty ]$. \newline
(a) If $f$ is non-contractive, then 
\begin{equation*}
\begin{aligned} C(E,f,\Gamma )& =\left\{ (x_{j})_{j=1}^{\infty }\in
E:(f(x_{j}))_{j=1}^{\infty }\notin \textstyle\bigcup\limits_{q\in \Gamma
}\ell _{q}(Y)\right\} , \\ C(E,f,0) & =\left\{ (x_{j})_{j=1}^{\infty }\in
E:(f(x_{j}))_{j=1}^{\infty }\notin c_{0}(Y)\right\} \end{aligned}
\end{equation*}%
are either empty or spaceable in $E$ (i.e., the union of the set with $\{0\}$
contains a closed infinite-dimensional subspace of $E$). \newline
(b) If $f$ is strongly non-contractive, then \newline
\begin{equation*}
C^{w}(E,f,\Gamma )=\left\{ (x_{j})_{j=1}^{\infty }\in
E:(f(x_{j}))_{j=1}^{\infty }\notin \textstyle\bigcup\limits_{q\in \Gamma
}\ell _{q}^{w}(Y)\right\}
\end{equation*}%
is either empty or spaceable in $E$.
\end{theorem}

In the present paper we formally extend Theorem 1.4 to a more general
setting.

\section{Spaceability and strongly invariant sequence spaces}

From now on an invariant sequence space $E$ over a Banach space $X$ will be
called strongly invariant sequence space when:

(a) $c_{00}$ $\left( X\right) \subset E;$

(b) $(x_{j})_{j=1}^{\infty }\in E$ if, and only if, all subsequences of $%
\left( x_{j}\right) _{j=1}^{\infty }$ also belong to $E.$



\begin{example}
\label{8t}If $X$ is a Banach space, then $\ell _{q}(X),\ell _{q}^{u}(X),\ell
_{q}^{w}(X),c(X),c_{0}(X)$ are strongly invariant sequence spaces.
\end{example}

The notion of strongly invariant sequence space is quite natural, but the
following example shows that there exist invariant sequence spaces which are
not strongly invariant sequence spaces:

\begin{example}
The Banach space $E=\left\{ \left( x_{j}\right)_{j=1}^{\infty} \in \ell
_{\infty }: x_{2n-1}= x_{2n} \text{ for all positive integers }n\right\} ,$
with the supremum norm, is an invariant sequence space but is not an
strongly invariant sequence space.
\end{example}

The following definition shall be used in the statement of our first main
result and in Section 3.

\begin{definition}
\label{compseqspa} \bigskip Let $X,Y$ be Banach spaces and $E$ be an
invariant sequence space over $Y$. A map $f:X\rightarrow $ $Y$ such that $%
f(0)=0$ is said to be compatible with $E$ if for any sequence $\left(
x_{j}\right) _{j=1}^{\infty }$ of \ elements of $X$, we have 
\begin{equation*}
\left( f(x_{j}\right) )_{j=1}^{\infty }\notin E\Rightarrow \left(
f(ax_{j}\right) )_{j=1}^{\infty }\notin E
\end{equation*}%
regardless of the scalar $a\neq 0.$
\end{definition}

\begin{example}
\label{12}Any non-contractive mapping $f:X\rightarrow Y$ is compatible with $%
\ell _{q}(Y)$ and $c_{0}(Y)$.
\end{example}

Now we state and prove one of the the main results of this paper. We show
that \cite[Theorem 2.5 (a)]{favv} can be formally generalized to a more
general setting. The proof is an abstraction of the proof of Theorem 1.4(a):

\begin{theorem}
\label{999}Let $X$ and $Y$ be Banach spaces, $\Gamma $ be an arbitrary set, $%
E$ be an invariant sequence space over $X$ and $E_{l}$ be strongly invariant
sequence spaces over $Y$ for all $l$ in $\Gamma $. If $f\colon
X\longrightarrow Y$ is compatible with $E_{l}$ for all $l\in \Gamma $, then $%
G(E,f,(E_{l})_{l\in \Gamma })$ is either empty or spaceable.
\end{theorem}

\begin{proof}
Assume that $G(E,f,(E_{l})_{l\in \Gamma })$ is non-void and consider $%
x=(x_{j})_{j=1}^{\infty }\in G(E,f,(E_{l})_{l\in \Gamma }).$ Note that there
are infinitely many indexes $j$ such that $x_{j}\neq 0$ because $f(0)=0$ and 
$c_{00}(Y)\subset E_{l}$ for all $l.$ We thus conclude that $x^{0}\neq 0$.
We shall first show that 
\begin{equation*}
x^{0}\in G(E,f,(E_{l})_{l\in \Gamma }).\newline
\end{equation*}%
Let $U:=\bigcup_{l\in \Gamma }E_{l}$. We know that $(f(x_{j}))_{j=1}^{\infty
}\notin U$, and thus $[(f(x_{j}))_{j=1}^{\infty }]^{0}\notin U$, since $%
E_{l} $, for each $l$, is an invariant sequence space. Denote $%
x^{0}=(x_{j_{k}})_{k=1}^{\infty },$ where $x_{j_{k}}$ is the $k$-th non null
coordinate of $x$. Then, we shall show that $(f(x_{j_{k}}))_{k=1}^{\infty
}\notin U$. Since $f(0)=0,$ it follows that 
\begin{equation*}
\lbrack (f(x_{j_{k}}))_{k=1}^{\infty }]^{0}=[(f(x_{j}))_{j=1}^{\infty
}]^{0}\notin U.
\end{equation*}%
Hence, $(f(x_{j_{k}}))_{k=1}^{\infty }\notin U$ and thus $x^{0}\in
G(E,f,(E_{l})_{l\in \Gamma })$. \ As usual, let us split $\mathbb{N}$ as a
countable union of pairwise disjoint subsets $(\mathbb{N}_{i})_{i=1}^{\infty
}$ of $\mathbb{N}$ and we denote $\mathbb{N}_{i}=\{i_{1}<i_{2}<...\}.$
Consider%
\begin{equation*}
y_{i}=\sum_{k=1}^{\infty }x_{j_{k}}\otimes e_{i_{k}}\in X^{\mathbb{N}}.
\end{equation*}%
Observe that $y_{i}^{0}=x^{0};$ we thus have $0\neq y_{i}^{0}\in E$ for all $%
i$. Since $E$ is an invariant sequence space, it follows that $y_{i}\in E$
for all $i\in \mathbb{N}$. Note also that the set $\{y_{1},y_{2},...\}$ is
linearly independent. Moreover, $y_{i}\in G(E,f,(E_{l})_{l\in \Gamma })$. \
In fact, if $y_{i}=(y_{m}^{i})_{m=1}^{\infty },$ then 
\begin{equation*}
\lbrack (f(y_{m}^{i}))_{m=1}^{\infty }]^{0}=[(f(x_{j}))_{j=1}^{\infty
}]^{0}\notin E_{l}
\end{equation*}%
for each $i\in \mathbb{N}$ and $l\in \Gamma $. Let $K$ be the constant from
Definition \ref{seqspa} (b1) and consider $\tilde{s}=1$ if $E$ is a Banach
space and $\tilde{s}=s$ if $E$ is an $s$-Banach space, $0<s<1$. For $%
(a_{i})_{i=1}^{\infty }\in \ell _{\tilde{s}}$,

\begin{align*}
\sum_{i=1}^{\infty }\Vert a_{i}y_{i}\Vert _{E}^{\tilde{s}}&
=\sum_{i=1}^{\infty }|a_{i}|^{\tilde{s}}\cdot \Vert y_{i}\Vert _{E}^{\tilde{s%
}}\leq K^{\tilde{s}}\cdot \sum_{i=1}^{\infty }\left\vert a_{i}\right\vert ^{%
\tilde{s}}\cdot \left\Vert y_{i}^{0}\right\Vert _{E}^{\tilde{s}} \\
& =K^{\tilde{s}}\cdot \left\Vert x^{0}\right\Vert _{E}^{\tilde{s}}\cdot
\sum_{i=1}^{\infty }\left\vert a_{i}\right\vert ^{\tilde{s}}=K^{\tilde{s}%
}\cdot \left\Vert x^{0}\right\Vert _{E}^{\tilde{s}}\cdot \left\Vert
(a_{i})_{i=1}^{\infty }\right\Vert _{\tilde{s}}^{\tilde{s}}<\infty .
\end{align*}%
Then $\sum\limits_{i=1}^{\infty }\Vert a_{i}y_{i}\Vert^{\tilde{s}}
_{E}<\infty $ and in any case we conclude that $\sum\limits_{i=1}^{\infty
}a_{i}y_{i}$ converges in $E.$ Thus the operator%
\begin{equation*}
T\colon \ell _{\tilde{s}}\longrightarrow E~~,~~T\left( \left( a_{i}\right)
_{i=1}^{\infty }\right) =\sum\limits_{i=1}^{\infty }a_{i}y_{i},
\end{equation*}%
is well-defined, linear and injective. Let us show that $\overline{T\left(
\ell _{\tilde{s}}\right) }$ belongs to $G(E,f,(E_{l})_{l\in \Gamma })$.

Let us recall that $y_{i}=\sum_{k=1}^{\infty }x_{j_{k}}\otimes e_{i_{k}}\in
X^{\mathbb{N}}$ where $x_{j_{k}}$ is the $k$-th non-zero coordinate $%
x=(x_{j})_{j=1}^{\infty }\in G(E,f,(E_{l})_{l\in \Gamma })$. We shall show
that if $z=(z_{n})_{n=1}^{\infty }\in $ $\overline{T\left( \ell _{\tilde{s}%
}\right) }$ is a non null sequence then $\left( f(z_{n}\right)
)_{j=1}^{\infty }\notin \bigcup_{l\in \Gamma }E_{l}.$ There are sequences $%
\left( a_{i}^{(k)}\right) _{i=1}^{\infty }\in \ell _{\tilde{s}}$, $k\in 
\mathbb{N}$, such that $z=\lim_{k\rightarrow \infty }T\left( \left(
a_{i}^{(k)}\right) _{i=1}^{\infty }\right) $ in $E.$ Note that, for each $%
k\in \mathbb{N}$,%
\begin{equation*}
T\left( \left( a_{i}^{(k)}\right) _{i=1}^{\infty }\right)
=\sum\limits_{i=1}^{\infty }a_{i}^{(k)}y_{i}=\sum\limits_{i=1}^{\infty
}a_{i}^{(k)}\cdot \sum\limits_{p=1}^{\infty }x_{j_{p}}\otimes
e_{i_{p}}=\sum\limits_{i=1}^{\infty }\sum\limits_{p=1}^{\infty
}a_{i}^{(k)}x_{j_{p}}\otimes e_{i_{p}}.
\end{equation*}%
Since $z\neq 0$, let $r\in \mathbb{N}$ be such that $z_{r}\neq 0.$ Since $%
\mathbb{N}=\bigcup\limits_{j=1}^{\infty }\mathbb{N}_{j}$, there exist unique 
$m,t\in \mathbb{N}$ such that $e_{m_{t}}=e_{r}$. Thus, for each $k\in 
\mathbb{N}$, the $r$-th coordinate of $T\left( \left( a_{i}^{(k)}\right)
_{i=1}^{\infty }\right) $ is the vector $a_{m}^{(k)}x_{j_{t}}.$ The
condition \ref{seqspa}(b2) of Definition 1.1 assures that convergence in $E$
implies coordinatewise convergence, so 
\begin{equation*}
z_{r}=\lim_{k\rightarrow \infty }a_{m}^{(k)}x_{j_{t}}=\left(
\lim_{k\rightarrow \infty }a_{m}^{(k)}\right) x_{j_{t}}.
\end{equation*}%
It follows that $\alpha _{m}:=\lim\limits_{k\rightarrow \infty
}a_{m}^{(k)}\neq 0$. On the one hand we have 
\begin{equation*}
\alpha _{m}x_{j_{p}}=\left( \lim_{k\rightarrow \infty }a_{m}^{(k)}\right)
x_{j_{p}}=\lim_{k\rightarrow \infty }a_{m}^{(k)}x_{j_{p}}
\end{equation*}%
for every $p\in \mathbb{N}$. On the other hand, for $p,k\in \mathbb{N}$, the 
$m_{p}$-th coordinate of $T\left( \left( a_{i}^{(k)}\right) _{i=1}^{\infty
}\right) $ is $a_{m}^{(k)}x_{j_{p}}.$ So, coordinatewise convergence gives $%
\lim\limits_{k\rightarrow \infty }a_{m}^{(k)}x_{j_{p}}=z_{m_{p}}$. It
follows that $z_{m_{p}}=\alpha _{m}x_{j_{p}}$ for every $p\in \mathbb{N}$.
As $\left( f(z_{m_{p}})\right) _{p=1}^{\infty }=\left( f\left(
a_{m}x_{j_{p}}\right) \right) _{p=1}^{\infty }$ and $\left( f\left(
x_{j_{p}}\right) \right) _{p=1}^{\infty }\notin E_{l}$, for all $l\in \Gamma 
$, by Definition \ref{compseqspa}, it follows that $\left(
f(z_{m_{p}})\right) _{p=1}^{\infty }\notin E_{l}$, for all $l\in \Gamma .$
Since $\left( f(z_{m_{p}})\right) _{p=1}^{\infty }$ is a subsequence of $%
\left( f(z_{n})\right) _{n=1}^{\infty }$ and $E_{l},$ for each $l\in \Gamma
, $ is a strongly invariant sequence space, it follows that 
\begin{equation*}
\left( f(z_{j}\right) )_{j=1}^{\infty }\notin E_{l},
\end{equation*}%
for all $l\in \Gamma ,$ and it completes the proof that $z\in
G(E,f,(E_{l})_{l\in \Gamma })$.
\end{proof}

From the previous theorem and Examples \ref{8t} and \ref{12} we have the
following corollary that recovers (\cite[Theorem 2.5 (a)]{favv}):

\begin{corollary}
\bigskip \textrm{\textrm{(\cite{favv})}} Let $X$ and $Y$ be Banach spaces, $%
E $ be an invariant sequence space over $X$, $f\colon X\longrightarrow Y$ be
a function and $\Gamma \subseteq (0,\infty ]$. If $f$ is non-contractive,
then $C(E,f,\Gamma )$ and $C(E,f,0)$ are either empty or spaceable.
\end{corollary}

The next immediate corollary of Theorem \ref{999} shows that the \cite[%
Corollaries 2.7, 2.8 and 2.10]{favv} and \cite[Theorem 1.3]{bdf} are all
particular cases of the following general result:

\begin{corollary}
Let $X$ and $Y$ be Banach spaces. Let $E$ be an invariant sequence space
over $X$ and $F$ be an strongly invariant sequence space over $Y.$ If $%
f\colon X\longrightarrow Y$ is compatible with $F$ and the set 
\begin{equation*}
A:=\left\{ (x_{j})_{j=1}^{\infty }\in E:(f(x_{j}))_{j=1}^{\infty }\notin
F\right\}
\end{equation*}%
is non empty, then $A$ is spaceable in $E$.
\end{corollary}

\section{The \textquotedblleft weak\textquotedblright\ case}

In this section we prove an extension of \cite[Theorem 2.5(b)]{favv}, i.e.,
an extension of Theorem 1.4(b) to more general invariant sequence spaces.

\label{y6}Let $F$ be an invariant sequence space over $\mathbb{K}$. For any
Banach space $Y$ we define%
\begin{equation*}
F^{w}(Y):=\{(x_{j})_{j=1}^{\infty }\in Y^{\mathbb{N}}:(\varphi
(x_{j}))_{j=1}^{\infty }\in F\text{ for all }\varphi \in Y^{\prime }\}.
\end{equation*}%
It is interesting to remark that if $F$ is an invariant sequence space over $%
\mathbb{K}$, then 
\begin{equation}
\sup_{\left\Vert \varphi \right\Vert \leq 1}\left\Vert (\varphi
(x_{j}))_{j=1}^{\infty }\right\Vert _{F}<\infty   \label{u09}
\end{equation}%
for all $(x_{j})_{j=1}^{\infty }\in F^{w}(Y).$ In fact, if $%
(x_{j})_{j=1}^{\infty }\in F^{w}(Y)$ we consider%
\begin{eqnarray*}
u &:&Y^{\prime }\rightarrow F \\
\varphi  &\mapsto &(\varphi (x_{j}))_{j=1}^{\infty }.
\end{eqnarray*}%
and a general version of the Closed Graph Theorem to topological vector
spaces (see \cite[page 51]{rudin}) finish the proof that $u$ is continuous.
Indeed, if%
\begin{equation*}
\varphi _{n}\rightarrow \varphi _{0}\text{ and }u(\varphi _{n})\rightarrow
\left( z_{j}\right) _{j=1}^{\infty }\in F
\end{equation*}%
we shall show that $\left( z_{j}\right) _{j=1}^{\infty }=u\left( \varphi
_{0}\right) .$ Since $\varphi _{n}\rightarrow \varphi _{0}$ we have $\varphi
_{n}(x_{j})\rightarrow \varphi _{0}(x_{j})$ for all $j.$ On the other hand,
since $u(\varphi _{n})\rightarrow \left( z_{j}\right) _{j=1}^{\infty }$ in $F
$, i.e., $(\varphi _{n}(x_{j}))_{j=1}^{\infty }\rightarrow \left(
z_{j}\right) _{j=1}^{\infty }$ in $F$, we also have $\varphi
_{n}(x_{j})\rightarrow z_{j}$ for all $j$. Therefore 
\begin{equation*}
\varphi _{0}(x_{j})=z_{j}
\end{equation*}%
for all $j,$ and hence 
\begin{equation*}
\left( z_{j}\right) _{j=1}^{\infty }=u\left( \varphi _{0}\right) .
\end{equation*}

\bigskip The next simple lemma highlights some tools to the proof of the
main result of this section.

\begin{lemma}
\label{hj}Let $Y$ be a Banach space. If $F$ is an strongly invariant
sequence space then 
\begin{equation}
x^{0}\in F^{w}(Y)\Leftrightarrow x\in F^{w}(Y)  \label{8gg}
\end{equation}%
and%
\begin{equation}
c_{00}(Y)\subset F^{w}(Y)  \label{t5}
\end{equation}
\end{lemma}

\begin{proof}
If $x=\left( x_{j}\right) _{j=1}^{\infty }$ and $x^{0}=\left(
x_{j_{k}}\right) _{k=1}^{\infty }$, then 
\begin{eqnarray*}
x &\in &F^{w}(Y)\Leftrightarrow \left( \varphi \left( x_{j}\right) \right)
_{j=1}^{\infty }\in F\text{ for all }\varphi \in Y^{\prime } \\
&\Leftrightarrow &\text{ }\left( \left( \varphi \left( x_{j}\right) \right)
_{j=1}^{\infty }\right) ^{0}\in F\text{ for all }\varphi \in Y^{\prime } \\
&\Leftrightarrow &\text{ }\left( \left( \varphi \left( x_{j_{k}}\right)
\right) _{k=1}^{\infty }\right) ^{0}\in F\text{ for all }\varphi \in
Y^{\prime } \\
&\Leftrightarrow &\text{ }\left( \left( \varphi \left( x_{j_{k}}\right)
\right) _{k=1}^{\infty }\right) \in F\text{ for all }\varphi \in Y^{\prime }
\\
&\Leftrightarrow &x^{0}=\left( x_{j_{k}}\right) _{k=1}^{\infty }\in F^{w}(Y).
\end{eqnarray*}%
The proof of (\ref{t5}) is a straightforward consequence of the fact that $F$
is an strongly invariant sequence space.
\end{proof}

\begin{example}
For $F=\ell _{p},c,c_{0},$ the respective $F^{w}(Y)$ are the well-known
invariant sequence spaces $\ell _{p}^{w}(Y),c^{w}(Y),c_{0}^{w}(Y).$
\end{example}

\begin{definition}
Let $X$ and $Y$ be Banach spaces, and $F$ be an invariant sequence space
over $\mathbb{K}$. A map $f:X\rightarrow Y$ such that $f(0)=0$ is strongly
compatible with $F^{w}(Y)$ if $\varphi \circ f$ is compatible with $F$ for
all continuous linear functionals $\varphi :Y\rightarrow \mathbb{K}$.
\end{definition}

\begin{example}
Any strongly non-contractive mapping (see (\ref{dd44})) $f:X\rightarrow Y$
is strongly compatible with $\ell _{q}^{w}(Y)$ and $c_{0}^{w}(Y)$.
\end{example}

\begin{definition}
Let $X$ and $Y$ be Banach spaces, $\Gamma $ be an arbitrary set and $E$ be
an invariant sequence space over $X.$ If $F_{l},$ for all $l\in \Gamma ,$
are invariant sequence spaces over $\mathbb{K}$, and $f:X\longrightarrow Y$
is any map, we define the set%
\begin{equation*}
G^{w}(E,f,(F_{l})_{l\in \Gamma })=\left\{ (x_{j})_{j=1}^{\infty }\in
E:(f(x_{j}))_{j=1}^{\infty }\notin \bigcup_{l\in \Gamma
}F_{l}^{w}(Y)\right\} .
\end{equation*}
\end{definition}

\bigskip The following theorem is a formal generalization of Theorem 1.4(b).
The proof follows the lines of the previous proofs but we present the
details for the sake of completeness:

\begin{theorem}
\label{9993333}Let $X$ and $Y$ be Banach spaces, $\Gamma $ be an arbitrary
set, $E$ be an invariant sequence space over $X$ and $F_{l}$ be strongly
invariant sequence spaces over $\mathbb{K}$ for all $l\in \Gamma $. If $%
f\colon X\longrightarrow Y$ is strongly compatible with $F_{l}^{w}(Y)$ for
all $l\in \Gamma $, then $G^{w}(E,f,(F_{l})_{l\in \Gamma })$ is either empty
or spaceable.
\end{theorem}

\bigskip

\begin{proof}
Assume that $G^{w}(E,f,(F_{l})_{l\in \Gamma })$ is non-void and consider $%
x=(x_{j})_{j=1}^{\infty }\in G^{w}(E,f,(F_{l})_{l\in \Gamma }).$ As in the
previous proofs, we shall begin by showing that 
\begin{equation*}
x^{0}\in G^{w}(E,f,(F_{l})_{l\in \Gamma }).\newline
\end{equation*}

Note that $x_{j}\neq 0$ for infinitely many $j$, because $f(0)=0$ and from (%
\ref{t5}) we have $c_{00}(Y)\subset F_{l}^{w}(Y).$

Denote $U=\bigcup_{l\in \Gamma }F_{l}^{w}(Y)$. We know that for each $l$
there is a $\varphi _{l}$ such that 
\begin{equation}
(\varphi _{l}\circ f(x_{j}))_{j=1}^{\infty }\notin F_{l},  \label{mna}
\end{equation}%
and thus 
\begin{equation}
\lbrack (\varphi _{l}\circ f(x_{j}))_{j=1}^{\infty }]^{0}\notin F_{l},
\label{986}
\end{equation}%
for all $l$, because $F_{l}$, for each $l$, is an invariant sequence space.
Denote $x^{0}=(x_{j_{k}})_{k=1}^{\infty },$ where $x_{j_{k}}$ is the $k$-th
non null coordinate of $x$. Now, we shall show that 
\begin{equation}
(\varphi _{l}\circ f(x_{j_{k}}))_{k=1}^{\infty }\notin F_{l}  \label{mnv}
\end{equation}%
for all $l$. Suppose that 
\begin{equation}
(\varphi _{l_{0}}\circ f(x_{j_{k}}))_{k=1}^{\infty }\in F_{l_{0}}
\label{mnb}
\end{equation}%
for some $l_{0}$. From (\ref{mnb}), since $F_{l_{0}}$ $\ $is an invariant
sequence space, we would have $[(\varphi _{l_{0}}\circ
f(x_{j_{k}}))_{k=1}^{\infty }]^{0}\in F_{l_{0}}$. But, since $\varphi
_{l_{0}}\circ f(0)=0,$ it would follow from (\ref{986}) that 
\begin{equation*}
\lbrack (\varphi _{l_{0}}\circ f(x_{j_{k}}))_{k=1}^{\infty }]^{0}=[(\varphi
_{l_{0}}\circ f(x_{j}))_{j=1}^{\infty }]^{0}\notin F_{l_{0}}.
\end{equation*}%
Since $F_{l_{0}}$ is an invariant sequence space we would have%
\begin{equation*}
(\varphi _{l_{0}}\circ f(x_{j_{k}}))_{k=1}^{\infty }\notin F_{l_{0}}
\end{equation*}%
and this contradicts (\ref{mnb}). Therefore we have (\ref{mnv}), i.e., 
\begin{equation*}
(f(x_{j_{k}}))_{k=1}^{\infty }\notin F_{l}^{w}(Y)
\end{equation*}%
for all $l$, and thus 
\begin{equation*}
x^{0}\in G^{w}(E,f,(F_{l})_{l\in \Gamma }).
\end{equation*}%
\ Again, let us separate $\mathbb{N}$ as a countable union of pairwise
disjoint subsets $(\mathbb{N}_{i})_{i=1}^{\infty }$ of $\mathbb{N}$ and as
usual, for all $i,$ we represent $\mathbb{N}_{i}=\{i_{1}<i_{2}<...\}.$
Consider%
\begin{equation*}
y_{i}=\sum_{k=1}^{\infty }x_{j_{k}}\otimes e_{i_{k}}\in X^{\mathbb{N}}.
\end{equation*}%
Since $E$ is an invariant sequence space, it follows that $y_{i}\in E$ for
all $i\in \mathbb{N}$. It is plain that $\{y_{1},y_{2},...\}$ is linearly
independent and $y_{i}^{0}=x^{0};$ we thus have $0\neq y_{i}^{0}\in E$ for
all $i$. Moreover, $y_{i}\in G^{w}(E,f,(F_{l})_{l\in \Gamma })$. In fact, if 
$y_{i}=(y_{m}^{i})_{m=1}^{\infty },$ then 
\begin{equation*}
\lbrack (f(y_{m}^{i}))_{m=1}^{\infty }]^{0}=[(f(x_{j}))_{j=1}^{\infty
}]^{0}\notin F_{l}^{w}(Y)
\end{equation*}%
for each $l$. Therefore, from (\ref{8gg}) of Lemma \ref{hj}, we have 
\begin{equation*}
(f(y_{m}^{i}))_{m=1}^{\infty }\notin F_{l}^{w}(Y)
\end{equation*}%
for each $i\in \mathbb{N}$ and $l\in \Gamma $ and thus $y_{i}\in
G^{w}(E,f,(F_{l})_{l\in \Gamma })$. Let $\tilde{s}=1$ if $E$ is a Banach
space and $\tilde{s}=s$ if $E$ is an $s$-Banach space, $0<s<1$. Proceeding
as in the proof of Theorem \ref{999} we know that the operator%
\begin{equation*}
T\colon \ell _{\tilde{s}}\longrightarrow E~~,~~T\left( \left( a_{i}\right)
_{i=1}^{\infty }\right) =\sum\limits_{i=1}^{\infty }a_{i}y_{i},
\end{equation*}%
is well-defined and injective. It remains to show that $\overline{T\left(
\ell _{\tilde{s}}\right) }$ belongs to $G^{w}(E,f,(F_{l})_{l\in \Gamma })$.
We shall show that if $z=(z_{n})_{n=1}^{\infty }\in $ $\overline{T\left(
\ell _{\tilde{s}}\right) }$ is a non null sequence then $\left(
f(z_{n}\right) )_{j=1}^{\infty }\notin \bigcup_{l\in \Gamma }F_{l}^{w}(Y).$
Let $\left( a_{i}^{(k)}\right) _{i=1}^{\infty }\in \ell _{\tilde{s}}$, $k\in 
\mathbb{N}$, be such that $z=\lim_{k\rightarrow \infty }T\left( \left(
a_{i}^{(k)}\right) _{i=1}^{\infty }\right) $ in $E.$ Note that, for each $%
k\in \mathbb{N}$,%
\begin{equation*}
T\left( \left( a_{i}^{(k)}\right) _{i=1}^{\infty }\right)
=\sum\limits_{i=1}^{\infty }a_{i}^{(k)}y_{i}=\sum\limits_{i=1}^{\infty
}a_{i}^{(k)}\cdot \sum\limits_{p=1}^{\infty }x_{j_{p}}\otimes
e_{i_{p}}=\sum\limits_{i=1}^{\infty }\sum\limits_{p=1}^{\infty
}a_{i}^{(k)}x_{j_{p}}\otimes e_{i_{p}}.
\end{equation*}%
Fix $r\in \mathbb{N}$ such that $z_{r}\neq 0.$ Since $\mathbb{N}%
=\bigcup\limits_{j=1}^{\infty }\mathbb{N}_{j}$, there are (unique) $m,t\in 
\mathbb{N}$ such that $e_{m_{t}}=e_{r}$. Thus, for each $k\in \mathbb{N}$,
the $r$-th coordinate of $T\left( \left( a_{i}^{(k)}\right) _{i=1}^{\infty
}\right) $ is the vector $a_{m}^{(k)}x_{j_{t}}.$ From the Definition \ref%
{seqspa}(b2) we know that convergence in $E$ implies coordinatewise
convergence, and thus 
\begin{equation*}
z_{r}=\lim_{k\rightarrow \infty }a_{m}^{(k)}x_{j_{t}}=\left(
\lim_{k\rightarrow \infty }a_{m}^{(k)}\right) x_{j_{t}}.
\end{equation*}%
It follows that $\alpha _{m}:=\lim\limits_{k\rightarrow \infty
}a_{m}^{(k)}\neq 0$ and 
\begin{equation*}
\alpha _{m}x_{j_{p}}=\left( \lim_{k\rightarrow \infty }a_{m}^{(k)}\right)
x_{j_{p}}=\lim_{k\rightarrow \infty }a_{m}^{(k)}x_{j_{p}}
\end{equation*}%
for every $p\in \mathbb{N}$. Besides, for $p,k\in \mathbb{N}$, the $m_{p}$%
-th coordinate of $T\left( \left( a_{i}^{(k)}\right) _{i=1}^{\infty }\right) 
$ is $a_{m}^{(k)}x_{j_{p}}.$ So, coordinatewise convergence gives $%
\lim\limits_{k\rightarrow \infty }a_{m}^{(k)}x_{j_{p}}=z_{m_{p}}$ and hence%
\begin{equation*}
z_{m_{p}}=\alpha _{m}x_{j_{p}}
\end{equation*}%
for every $p\in \mathbb{N}$. \ Therefore 
\begin{equation}
\left( \varphi _{l}\circ f(z_{m_{p}})\right) _{p=1}^{\infty }=\left( \varphi
_{l}\circ f\left( a_{m}x_{j_{p}}\right) \right) _{p=1}^{\infty }
\label{aaaa}
\end{equation}%
Since $f$ is strongly compatible with $F_{l}^{w}(Y)$ for all $l$, we
conclude that $\varphi _{l}\circ f$ is compatible with $F_{l}$ and thus,
from (\ref{mnv}), it follows that%
\begin{equation}
\left( \varphi _{l}\circ f\left( a_{m}x_{j_{p}}\right) \right)
_{p=1}^{\infty }\notin F_{l}.  \label{hyu}
\end{equation}%
Thus, from (\ref{aaaa}) and (\ref{hyu}) we have 
\begin{equation}
\left( \varphi _{l}\circ f(z_{m_{p}})\right) _{p=1}^{\infty }\notin F_{l},
\label{hhh}
\end{equation}%
for all $l$. Therefore since $\left( \varphi _{l}\circ f(z_{m_{p}})\right)
_{p=1}^{\infty }$ is a subsequence of $\left( \varphi _{l}\circ
f(z_{n})\right) _{n=1}^{\infty }$ and $F_{l}$ is a strongly invariant
sequence space, it follows that 
\begin{equation*}
\left( \varphi _{l}\circ f(z_{j}\right) )_{j=1}^{\infty }\notin F_{l},
\end{equation*}%
for all $l\in \Gamma ,$ and we finally conclude that $z\in
G^{w}(E,f,(F_{l})_{l\in \Gamma })$.
\end{proof}

\bigskip Let $X$ and $Y$ be Banach spaces, $E$ be an invariant sequence
space over $X$ and $F_{l}=\ell _{l}$, with $l\in \Gamma \subset (0,\infty ].$
If $f\colon X\longrightarrow Y$ is strongly non-contractive then $f$ is
strongly compatible with $\ell _{l}^{w}(Y)$ and from the previous theorem we
conclude that $G^{w}(E,f,(F_{l})_{l\in \Gamma })$ is either empty or
spaceable. Since 
\begin{equation*}
G^{w}(E,f,(F_{l})_{l\in \Gamma })=\left\{ (x_{j})_{j=1}^{\infty }\in
E:(f(x_{j}))_{j=1}^{\infty }\notin \textstyle\bigcup\limits_{l\in \Gamma
}\ell _{l}^{w}(Y)\right\}
\end{equation*}%
we recover Theorem 1.4(b).

\begin{remark}
It is interesting to note that in all the results presented in this note
(and in the respective versions from \cite{barr, favv}) the
lineability/spaceability results satisfy a somewhat slightly stronger
condition, in the following sense: given any point $x$\ of the set $%
G(E,f,(E_{l})_{l\in \Gamma })$ it is proved here that there is a closed
infinite dimensional vector space $V$ such that \textquotedblleft
essentially\textquotedblright\ $x\in V\subset G(E,f,(E_{l})_{l\in \Gamma
})\cup \{0\}$. This leads to the following extension of the notion of
lineability that may be interesting to investigate in different contexts: a
subset $A$ of a vector space $W$ is \textit{\ pointwise $\lambda $-lineable}
if for any $x\in A$ there is a $\lambda $-dimensional vector space $V$ such
that $x\in V\subset A\cup \{0\}\subset W.$ The same definition can be
adapted to the notion of spaceability. It is not difficult to verify that in
general these concepts are strictly stronger than just
lineability/spaceability. For instance, let $W=\ell _{2}$ and 
\begin{equation*}
A=\left( \mathrm{span}\{e_{1}\}\right) \cup \left( \mathrm{span}%
\{e_{2},e_{3}\}\right) \cup \left( \mathrm{span}\{e_{4},...,e_{6}\}\right)
\cup ...
\end{equation*}%
It is plain that $A$ is $n$-lineable for all positive integer $n$, but $A$
is not pointwise $2$-lineable.
\end{remark}

\textbf{Acknowledgment.} The authors thank the referee for important suggestions and corrections

\end{document}